\documentclass{article}
\usepackage{amsmath,amsthm,hyperref}
\usepackage{amsfonts,amssymb}
\usepackage[left=0.5in, right=0.5in, top=1in, bottom=1in]{geometry}
\usepackage{graphicx} % Required for inserting images
\usepackage[shortlabels]{enumitem}
\usepackage{xcolor}
\usepackage{makecell}
\newcommand{\ceil}[1]{\lceil #1 \rceil}

\newcommand{\ignore}[1]{}
\newcommand{\Z}{\mathbb{Z}}
\newcommand{\E}{\mathbb{E}}

\newtheorem{theorem}{Theorem}[section]

\newtheorem{lemma}[theorem]{Lemma}

\newtheorem{startseq}{Starting Sequence}
\newtheorem{rul}{Rule}
\newtheorem{corollary}[theorem]{Corollary}
\newtheorem{conjecture}[theorem]{Conjecture}

\theoremstyle{definition}
\newtheorem{definition}[theorem]{Definition}

\newtheorem*{remark}{Remark}

\theoremstyle{remark}

\numberwithin{equation}{section}
\numberwithin{case}{theorem}

\title{Weakly Consecutive Sequences}
\author{Thomas Garrison, Chris Seiler, Andrew Knowles}
\date{October 2023} % should this be December now haha

\begin{document}

\maketitle

\begin{abstract}
A \emph{weakly consecutive sequence (WCS)} is a permutation $\sigma$ of $\{1, \ldots, k\}$ such that if an integer $d$ divides $\sigma(i)$, then $d$ also divides $\sigma(i \pm d)$ insofar as these are defined. The structure of weakly consecutive sequences is surprisingly rich, and it is difficult to find a formula for the number $N(k)$ of WCS's of length $k$. However, for a given $k$ we describe four starting sequences, to each of which we can apply three \emph{rules} or operations to generate new WCS's. We conjecture that any WCS can be constructed by applying these rules, which depend in an intricate way on the primality of $k$ and surrounding integers. We find bounds for $N(k)$ by analyzing these rules.
% I previously used the word "forms" to describe the base cases, but maybe case is just a better word. Also thought about "classes." Ok, now we use starting sequences.
\end{abstract}
\section{Introduction}
% Define all major quants of interest,
% Major theorems (can generate all using rules described in next section
% outline paper organization
\begin{definition} \label{def:weak_consec}
    Let $\sigma : [k] \to [k]$ be a permutation for a natural number $k$. We say $\sigma$ is \emph{weakly consecutive} if, for all $i,j \in [k]$ and integers $m$, if $m \mid \sigma(i)$ and $m \mid (i - j)$, then $m \mid \sigma(j)$. Let $N(k)$ be the number of distinct weakly consecutive sequences of length $k$.
\end{definition}

Intuitively, a weakly consecutive sequence is a sequence where multiples of numbers are preserved in
order. That is, everything an even distance from the number 2 is also even; everything distance 3
away from 3 is also divisible by 3, etc. We call these sequences weakly consecutive because they share this property with the consecutive sequence $1,\ldots,k$.

For example, all weakly consecutive sequences of length up to $8$ are shown in Table \ref{tab:WCS}.
\begin{table}[htbp]
    \centering
    \begin{tabular}{c|l}
    Length & Sequences \\ 
    \hline
    1 & $(1)$ \\
    2 & $(1,2)$, $(2,1)$ \\
    3 & $(1,2,3)$, $(3,2,1)$ \\
    4 & $(1,2,3,4)$, $(1,4,3,2)$, $(2,3,4,1)$, $(4,3,2,1)$ \\
    5 & $(1,2,3,4,5)$, $(1,4,3,2,5)$, $(5,2,3,4,1)$, $(5,4,3,2,1)$ \\
    6 & $(1,2,3,4,5,6)$, $(1,6,5,4,3,2)$, $(2,3,4,5,6,1)$, $(6,5,4,3,2,1)$ \\
    7 & $(1,2,3,4,5,6,7)$, $(1,6,5,4,3,2,7)$, $(7,2,3,4,5,6,1)$, $(7,6,5,4,3,2,1)$ \\
    8 & $(1,2,3,4,5,6,7,8)$, $(1,2,3,8,5,6,7,4)$, $(4,7,6,5,8,3,2,1)$, $(8,7,6,5,4,3,2,1)$
\end{tabular}
    \caption{Weakly consecutive sequences of length up to $8$}
    \label{tab:WCS}
\end{table}

The values of $N(k)$ for $k=1, \dots, 100$ are shown in Table \ref{tab:N(k)}.
\begin{table}
    \centering
    \begin{tabular}{>{\bfseries}c||c|c|c|c|c|c|c|c|c|c}
        + & \bfseries{ 1 } & \bfseries{ 2 } & \bfseries{ 3 } & \bfseries{ 4 } & \bfseries{ 5 } & \bfseries{ 6 } & \bfseries{ 7 } & \bfseries{ 8 } & \bfseries{ 9 } & \bfseries{ 10 } \\
        \hline
        \hline
        0 & 1 & 2 & 2 & 4 & 4 & 4 & 4 & 4 & 8 & 16 \\ 
        \hline
        10 & 16 & 4 & 4 & 2 & 2 & 8 & 8 & 8 & 8 & 4 \\ 
        \hline
        20 & \textbf{12} & 16 & 16 & 2 & 4 & 4 & 8 & 16 & 16 & 8 \\ 
        \hline
        30 & 8 & 8 & 8 & 16 & 16 & 8 & 8 & 4 & 4 & 8 \\ 
        \hline
        40 & 8 & 8 & 8 & 4 & 4 & 8 & 8 & 2 & 4 & 4 \\ 
        \hline
        50 & 4 & 8 & 8 & 4 & 4 & 2 & 6 & 8 & 8 & 4 \\ 
        \hline
        60 & 4 & 2 & 2 & 4 & 4 & 8 & 8 & 4 & 4 & 8 \\ 
        \hline
        70 & 8 & 8 & 8 & 4 & 4 & 4 & 4 & 8 & 8 & 4 \\ 
        \hline
        80 & \textbf{40} & 32 & 32 & 8 & 8 & 8 & 8 & 16 & 16 & 8 \\ 
        \hline
        90 & 8 & 8 & 8 & 8 & 8 & 8 & 8 & 4 & 4 & 8
    \end{tabular}
    \caption{Values of $N(k)$, $1 \leq k \leq 100$. The value in the row labeled $10 i$ and the column labeled $j$ is $N(10 i + j)$. Note that only $N(21) = 12$ and $N(81) = 40$ are not powers of $2$ (as shown in bold).}
    \label{tab:N(k)}
\end{table}

Observing these early sequences and values of $N(k)$ raises questions about the generation and classification of weakly consecutive sequences. For example, it is easy to understand at a glance that $(1,2,3,4)$ and $(4,3,2,1)$ are weakly consecutive, but $(1,4,3,2)$ and $(1,2,3,8,5,6,7,4)$ are less obvious. Note that $N(k)$ is usually a power of $2$, but not always: we have $N(21) = 12$ and $N(81) = 40$. In this paper, we will show that $N(k)$ is an unbounded function of $k$:
\begin{theorem}[see Corollary \ref{cor:nk} below]
  Fix an $n \in \mathbb{Z}^+$. There exists a $k \leq 2^{n^{2 \lg(\lg (n)) (1 + o(1))}}$ with $N(k) \geq n$.
\end{theorem}

\subsection{Starting sequences and rules}
For a given length $k$, we have found up to four types of starting sequence, to which we can then apply three operations, or \emph{rules,} in order to generate new weakly consecutive sequences, depending on conditions involving $k$. Briefly, the starting sequences are as follows:
\begin{enumerate}
    \item Consecutive (for any $k$): $(1,\dots,k)$

    Example: $(1,2,3,4)$
    \item 1-Inversion ($k=p-1$ for prime $p$): $(2,\dots,k,1)$

    Example: $(2,3,4,5,6,\mathbf{1})$
    \item Twice Twin Prime ($k=2p$ or $2p+1$ for prime $p,p+2$): Swap positions of $2, 2p$ and $p,p+2$, resulting in sequence $(1, 2p, \dots, p+2, p+1, p, \dots, 2)$ or $(1, 2p, \dots, p+2, p+1, p, \dots, 2, k)$

    Example: $(1,\mathbf{10},3,4,\mathbf{7},6,\mathbf{5},8,9,\mathbf{2})$ or $(1,\mathbf{10},3,4,\mathbf{7},6,\mathbf{5},8,9,\mathbf{2},11)$
    \item Twin Sophie Germain ($k=2p-1$ for prime $p,2p+1,p+2$): Move $2,1$ to the upper end and swap $p,p+2$, resulting in sequence $(3,\dots, p+2, p+1, p, \dots, k, 2, 1)$

    Example: $(3,4,5,6,7,8,9,10,\mathbf{13},12,\mathbf{11},14,15,16,17,18,19,20,21,\mathbf{2},\mathbf{1})$
\end{enumerate}

From these starting sequences, we can apply the following three operations (rules) repeatedly to generate new weakly consecutive sequences, if conditions allow:
\begin{enumerate}
    \item Twin Prime Swap (requires Starting Sequence 4): We can swap $q,q+2$ if both are prime with $\ceil{\frac{k}{2}} < q \leq k$ and both are in their starting positions for Starting Sequence 4.

    Example: 
    
    $(3,4,5,6,7,8,9,10,13,12,11,14,15,16,\mathbf{17},18,\mathbf{19},20,21,2,1) \to \\
    (3,4,5,6,7,8,9,10,13,12,11,14,15,16,\mathbf{19},18,\mathbf{17},20,21,2,1)$
    \item Prime Power Swap (any starting sequence): We can swap $p^c$ and $p^{c-1}$ for prime $p$, integer $c \geq 1$ if $p^c \leq k < p^c + p^{c-1}$.

    Example: $(1,\mathbf{2},3,\mathbf{4},5) \to (1,\mathbf{4},3,\mathbf{2},5)$
    \item Trivial Reversal (any starting sequence): We can reverse any weakly consecutive sequence.

    Example: $(1,2,3,4,5) \to (5,4,3,2,1)$
\end{enumerate}

We will explore these starting sequences and rules, and their proofs, in more detail in Section \ref{sec:rules}.

% This part can be included here or in Prime Power Swaps section. Here, it requires we describe the rules first.

The Prime Power Swap is of particular interest due to its potential to produce large amounts of WCS's. In particular, we are interested in finding the least $k$ such that $n$ sequences can be obtained by solely applying Prime Power Swaps, or in other words, the least $k$ for which $n$ Prime Power Swaps can be performed maintaining weak consecutiveness. For example, the least $k$ for $n$ swaps up to $n=8$ is shown in Table \ref{tab:PPS}. Observe that the values of $k$ grow rapidly and rather erratically. We will study Prime Power Swaps in more detail in Section \ref{sec:primepowerswaps}.

The Twin Prime Swap has the potential to produce even larger amounts of WCS's for select values of $k$, but its application is conditional on the existence of many twin primes. We will discuss this further in Section \ref{sec:further}.

\begin{table}
\centering
\begin{tabular}{c|l}
    $n$ (Number of swaps) & Least $k$ (Length of sequence) \\ 
    \hline
    1 & 4 \\
    2 & 9 \\
    3 & 128 \\
    4 & 2209 \\
    5 & 4897369 \\
    6 & 1364785249 \\
    7 & 23995037731729 \\
    8 & 35278099774369 \\
\end{tabular}
\caption{The smallest $k$ admitting $n$ Prime Power Swaps, yielding at least $2^n$ WCS's}
\label{tab:PPS}
\end{table}

% We will explore this problem and others related to Prime Power Swaps in more detail in section \ref{sec:primepowerswaps}.

\section{General results}

In order to prove the validity of some of these starting sequences and rules, we require a certain property of weakly consecutive sequences involving the following sets:

\begin{definition}\label{def:characterization}
    Given a permutation $\sigma$ over $[k]$ and an $m \in [k]$, define the sets of indices $S_m^\sigma = \{ i \in [k] : i \equiv \sigma^{-1}(m) \bmod m \}$ and $T_m^\sigma = \{ i \in [k] : m \mid \sigma(i) \}$.
\end{definition}

The set $S_m^\sigma$ can be intuitively understood as all the indices lying a multiple of $m$ away from index $\sigma^{-1}(m)$, while $T_m^\sigma$ is the set of indices $i$ whose value $\sigma(i)$ is a multiple of $m$.

We prove the equivalence of these sets for weakly consecutive sequences.

\begin{lemma}[Division slice characterization of weakly consecutive sequence]\label{lem:characterization}
     Let $\sigma$ be a permutation over $[k]$. Then, $\sigma$ is weakly consecutive if and only if $S_m^\sigma = T_m^\sigma$ for all $m \in [k]$.
\end{lemma}

\begin{proof}
    For the forward direction, suppose $\sigma$ is weakly consecutive; fix an $m \in [k]$ and let $S_m^\sigma, T_m^\sigma$ be as defined above. We need three pieces to show $S_m^\sigma = T_m^\sigma$: namely, $|S_m^\sigma| \geq \lfloor k/m \rfloor$; $S_m^\sigma \subseteq T_m^\sigma$; and $|T_m^\sigma| = \lfloor k/m \rfloor$. Together these imply the forward direction.

    To show $|S_m^\sigma| \geq \lfloor k/m \rfloor$, let $a = \min S_m^\sigma$ and $b = \max S_m^\sigma$. Since $S_m^\sigma$ consists of evenly spaced elements by a distance of $m$, we have $b = a + (|S_m^\sigma| - 1)m$. For the same reason, $b \geq k-m+1$ and $a \leq m$, otherwise another element of $S_m^\sigma$ would appear. Combining, we have
    \[
      k - m + 1 \leq b = a + (|S_m^\sigma| - 1)m \leq m + (|S_m^\sigma| - 1) m = m|S_m^\sigma|.
    \]
    Rearranging, we get $|S_m^\sigma| \geq \frac {k+1}{m} - 1$, and clearly $|S_m^\sigma| \geq \lceil \frac{k+1}{m} \rceil - 1$ since $|S_m^\sigma|$ is an integer. Thus $|S_m^\sigma| \geq \lfloor k/m \rfloor$. %  ceil((k+1)/m - 1) >= floor(k/m) is a simple argument but not one I want to do right now
    
    Second: $S_m^\sigma \subseteq T_m^\sigma$. To show this, let $i \in S_m^\sigma$ and $j = \sigma^{-1}(m)$. We know that $i \equiv j \bmod m$ by definition of $S_m^\sigma$, or equivalently, $m \mid (j-i)$. And since $\sigma(j) = m$, we have $m \mid \sigma(j)$; by the definition of weak consecutiveness, $m \mid \sigma(i)$ and thus $i \in T_m^\sigma$.
    
    Third: Notice that $T_m^\sigma$ locates the multiples of $m$ as produced by $\sigma$; but since $\sigma$ is a permutation, it locates all the multiples of $m$ which are at most $k$, meaning $|T_m^\sigma| = \lfloor k/m \rfloor$.
    \\
    \\
    For the backwards direction, suppose $S_m^\sigma = T_m^\sigma$ for all $m \in [k]$. Fix an $m, i, j \in [k]$ with $m \mid \sigma(i)$ and $m \mid (i-j)$. By definition, $i \in T_m^\sigma = S_m^\sigma$ and thus $i \equiv \sigma^{-1}(m) \bmod m$. Furthermore, since $m \mid (i - j)$, it follows that $j \equiv i \equiv \sigma^{-1}(m) \bmod m$ and thus  $j \in S_m^\sigma$. So $j \in T_m^\sigma$, meaning $m \mid \sigma(j)$, as desired.
\end{proof}

In other words, $\sigma$ is weakly consecutive if and only if all multiples of $m$ in $\sigma$ have indices which are spaced $m$ apart.

\begin{remark} 
    For any permutation $\sigma$ over $[k]$ and any $m, n \in [k]$ with $m \mid n$, $T_m^\sigma \supseteq T_n^\sigma$.
\end{remark}

The following results are not needed in the sequel, but it is hoped that they shed added light on the structure of WCS's.

\begin{lemma}
    Let $\sigma$ be a weakly consecutive sequence of length $k$. For any $x \in [k]$, $\sigma^{-1}(x)$ lies between $(k \bmod m) + 1$ and $k - (k \bmod m)$, inclusive. %[State in terms of least $\sigma^{-1}(d\mathbb{Z})$?]
\end{lemma}

\begin{proof}
If not, it is easy to show that
\[
  |S_m^\sigma| = \lfloor k/m \rfloor + 1 > \lfloor k/m \rfloor = |T_m^\sigma|,
\]
contradicting Lemma \ref{lem:characterization}.
% TODO It might be good to put in a more detailed proof later.
\end{proof}

% Moved to Section 5:
%\begin{conjecture}
%    Given a weakly consecutive sequence $\sigma$ of length %$k$, we have $\sigma(1)=1$ or $\sigma(k)=1$.
%\end{conjecture}

\begin{lemma} \label{lem:div_slice}
    Fix $d \in [m]$. Let $\ell$ be the least index such that $d | \sigma(\ell)$. Then the permutation
    \[
      D_d (\sigma) : \big[\lfloor m/d \rfloor\big] \to \big[\lfloor m/d \rfloor\big]
    \]
    defined by
    \[
      D_d(\sigma)(i) = \frac{\sigma\big(\ell + (i-1)d\big)}{d}
    \]
    is weakly consecutive.

    % For weakly consecutive sequence $\sigma$ of length $k$ and natural number $x \in [k]$, let $x_i$ be the $i$-th multiple of $x$ in $\sigma$. Then the subsequence $\{a_i\}$ such that $a_i = x_i / x$ is weakly consecutive.
\end{lemma}

We refer to $D_d(\sigma)$ as the ``$d$th division slice'' of $\sigma$, explaining the name of Lemma \ref{lem:characterization}.

\begin{proof} Let $\sigma$ be a weakly consecutive sequence. For some $d$, let $\ell$ be the least index such that $d\mid \sigma (\ell)$. Suppose for integer $m$ and indices $i,j$ that $m \mid D_d(\sigma)(i)$ and $m \mid (i-j)$. We prove $m \mid D_d(\sigma(j))$ as follows:

Multiplying through by $d$ yields both
$$md \mid D_d(\sigma)(i)d = \sigma(\ell + (i-1)d)$$
and
$$md \mid (i-j)d = (\ell - \ell) + (i - 1 - j + 1)d = (\ell + (i-1)d) - (\ell + (j-1)d)$$
Now since $\sigma$ is weakly consecutive, we can apply the definition directly:
$$md \mid \sigma(\ell + (i-1)d), md \mid (\ell + (i-1)d) - (\ell + (j-1)d) \Rightarrow md \mid \sigma(\ell + (j-1)d)$$
And dividing by $d$ yields $m \mid \frac{\sigma(\ell + (j-1)d)}{d} = D_d(\sigma)(j)$, as required.
\end{proof}

\section{Generating weakly consecutive sequences} \label{sec:rules}

As mentioned, we are quite interested in operations under which the set of weakly consecutive sequences is closed. We have found four distinct starting sequences and three ``rules'' or operations which can be applied to these sequences repeatedly and interchangeably to generate new WCS's. Each starting sequence and rule have different conditions for applicability depending on the length $k$ of the sequence.

We begin by describing and proving the validity of the starting sequences, starting with the simplest.
\begin{startseq}[Consecutive] \label{start:consecutive}
    For any $k$, the sequence $\sigma = 1,\dots, k$, given by $\sigma(i)=i$ for all $i$, is weakly consecutive.
\end{startseq}

The proof is obvious.

\begin{startseq}[$1$-Inversion] \label{start:nontriv}
    Let $k=p-1$ for prime $p$. The sequence $\sigma = 2,\dots, k, 1$, given by 
    $$\sigma(i) = \begin{cases} 
          1 & i = k \\
          i+1 & \text{otherwise}
       \end{cases}$$
    is weakly consecutive.
\end{startseq}
For example, taking $p = 7$ and $k=6$, we get the weakly consecutive sequence:
\begin{gather*}
    (2,3,4,5,6,\mathbf{1}).
\end{gather*}

\begin{proof}
    Suppose $m \mid \sigma(i)$ and $m \mid (i - j)$, with $i \neq j$. The $m=1$ case is trivial so we assume $m \geq 2$. The assumptions are only possible for $i < k$, as $m > 1$ cannot divide $\sigma(k) = 1$. Thus we know $1 < m < k$ and $\sigma(i) = i + 1$, meaning $m \mid (i+1)$. Since $m \mid (i-j)$ this implies $m \mid (j+1)$. Furthermore, $j \neq k$, since otherwise $m$ is a nontrivial divisor of $k+1$, which is prime. So $\sigma(j) = j+1$, and we're done. 
\end{proof}

% Intuitively, this case can be understood as appending $p$ to the end of a consecutive sequence of length $k=p-1$, performing a Prime Power Swap $S_{p,1}(\sigma)$ (swapping $1$ and $p$), and then removing $p$ from the sequence (adjusting indices accordingly).

%maybe rename
\begin{startseq}[Twice Twin Prime] \label{start:twin}
    Let $k=2p$ or $k=2p+1$ for a prime $p$ with $p+2$ also prime. Then the sequence $\sigma$ given by 
    $$\sigma(i) = \begin{cases} 
          2 & i = 2p \\
          2p & i = 2 \\
          p & i = p+2 \\
          p+2 & i = p \\
          i & \text{otherwise}
       \end{cases}$$
    is weakly consecutive.
    
    In other words, a consecutive sequence with the positions of $2,2p$ and $p,p+2$ swapped is weakly consecutive.
\end{startseq}

For example, since $p = 5$ and $7$ are twin primes, we can take $k = 10$ or $11$ to get the following WCS's (bold indicates numbers that are out of consecutive order):
\[
  (1,\mathbf{10},3,4,\mathbf{7},6,\mathbf{5},8,9,\mathbf{2}) \quad \text{and} (\quad 1,\mathbf{10},3,4,\mathbf{7},6,\mathbf{5},8,9,\mathbf{2},11).
\]

\begin{proof}[Proof of Starting Sequence \ref{start:twin}]
    Fix $m, i, j$ with $i \neq j$ and suppose $m \mid \sigma(i)$ and $m \mid (i - j)$; we wish to show $m \mid \sigma(j)$. Note this is trivial if $m = 1$, so suppose $m \geq 2$. We proceed in cases according to the definition of $\sigma(i)$. 

    Case I: $i = 2p$. Then, $\sigma(i) = 2$ and thus $m = 2$, so $2 \mid (2p - j)$. Thus $j$ is even. Since $p$ is an odd prime, $j \neq p$ and $j \neq p + 2$, so either $j = 2$ or $\sigma(j) = j$. If $j = 2$, then $\sigma(j) = 2p$ which is divisible by $2$; otherwise, since $j$ is even, $2 \mid \sigma(j)$, as desired.

    Case II: $i = 2$. Then, $\sigma(i) = 2p$ and so $m = 2$ or $m = p$. If $m = 2$, then $2 \mid (2 - j)$ and $j$ is even; as argued in case I, $\sigma(j)$ is even for all even $j$. So now suppose $m = p$. So $p \mid (2 - j)$, or equivalently, $j = ap + 2$ for some integer $a$. Since $k \leq 2p + 1$, then, $j = p + 2$. Then $\sigma(p+2) = p$ which is divisible by $p$, as desired.

    Case III: $i = p + 2$, meaning $\sigma(i) = p$ and thus $m = p$. Thus $p \mid (p+2 - j)$, which is only possible for $j = 2$ or $j \geq 2p + 2$. But since $k \leq 2p + 1$, we know $j = 2$. Thus $\sigma(j) = 2p$ which is divisible by $p$, as desired.

    Case IV: $i = p$, meaning $\sigma(i) = p + 2$. We know $p+2$ is prime so $m = p + 2$. Thus $p + 2 \mid (p - j)$, which is only possible for $j \geq 2p + 2$. Once again, $k \leq 2p + 1$, so this case never happens.

    Case V: $\sigma(i) = i$. Note $m \neq p$, as $p$ and $2p$ are the only multiples of $p$ less than $k$. Similarly, $m \neq p + 2$. If $m = 2$, then from previous cases, we know that even indices in $\sigma$ produce even elements. Otherwise, if $m \geq 3$, then $m \mid (i - j)$ and $m \mid i$, meaning $m \mid j$. But if $j$ has a factor which is not $p$, $p+2$, or $2$, then $\sigma(j) = j$ which is divisible by $m$, as desired.
\end{proof}

%Should we call this the 2-Inversion? or something
\begin{startseq}[Twin Sophie Germain] \label{start:germain}
    Let $k=2p-1$, where $p, 2p+1, p+2$ are all prime. The sequence $\sigma$ given by 
    $$\sigma(i) = \begin{cases} 
          1 & i = k \\
          2 & i = k-1 \\
          p & i = p \\
          p+2 & i = p-2 \\
          i+2 & \text{otherwise}
       \end{cases}$$
    is weakly consecutive.

    In other words, if we move $2,1$ to the end of the consecutive sequence and swap the positions of $p$ and $p+2$, we have a weakly consecutive sequence.
    
\end{startseq}

\begin{proof}[Proof of Starting Sequence~\ref{start:germain}]
    Fix $m, i, j$ with $i \neq j$, $m \mid \sigma(i)$ and $m \mid (i - j)$. Furthermore we can assume $m \geq 2$.

    First, we show that $2 \mid \sigma(a)$ if and only if $2 \mid a$. If $2 \mid \sigma(a)$, then $a$ is either $2p - 2$ or $a + 2$. In the first case, clearly $2 \mid a$; and in the second case, $2 \mid a + 2$ since $a$ is even. Similarly, if $2 \mid a$, then either $a = 2p - 2$ or $\sigma(a) = a + 2$, both of which are even when $a$ is.

    This means that the $m = 2$ case is completed; since if $2 \mid \sigma(i)$ and $2 \mid (i - j)$, then $i$ is even (and thus so is $j$) meaning $\sigma(j)$ is even. So we now suppose $m \geq 3$.

    Now, we proceed by casing on $\sigma(i)$; we can skip $i = k$ and $i = k - 1$, since the corresponding values of $\sigma$ have no divisors greater than $2$.

    Case I: $i = p$, so $\sigma(i) = p$. Then, $m = p$; however, this case never happens since $p \mid (p - j)$ implies $j \geq 2p > k$.

    Case II: $i = p - 2$, so $\sigma(i) = p + 2$. Then $m = p + 2$; since $p + 2$ is prime, once again, this case never happens since $j \geq 2p + 4 > k$.

    Case III: $\sigma(i) = i + 2$. Then, $m \mid i + 2$ and $m \mid (i - j)$, meaning $m \mid (j + 2)$. Furthermore, $m \notin \{p, p+2\}$, since all multiples of $p$ and $p+2$ in $\sigma$ are covered by previous casing on values of $i$. This means $j \neq p-2$ and $j \neq p$ (otherwise $m \mid p$ or $m \mid (p+2)$, impossible since they are both prime).

    Furthermore, note $j \neq 2p - 2$. This is because otherwise, $m \mid 2p$ for $m \notin \{ 2, p \}$. Similarly, because $2p+1$ is prime, $j \neq 2p-1$. So we conclude that $\sigma(j) = j + 2$, meaning $m \mid \sigma(j)$ as desired.
\end{proof}
For example for 21 we have the following: 
\begin{gather*} 3,4,5,6,7,8,9,10,\mathbf{13},12,\mathbf{11},14,15,16,17,18,19,20,21,\mathbf{2},\mathbf{1} 
\end{gather*}

We now move on to describe the rules for generating new sequences from these starting sequences. Each rule has conditions for when it can be applied, but only one requires a specific starting sequence. This rule is as follows: 
% This should perhaps read something like, let sigma be a case 4 starting sequence *or sequence gen'd from a case 4 starting sequence*. It might be worth defining "classes" of sequences depending on the starting sequence, and using that def for these proofs instead of just "starting sequence" to allow repeated applications to be discussed.
\begin{rul}[Twin Prime Swaps]\label{rul:germain}
   Suppose $q$ and $q+2$ are both prime with $\ceil{\frac{k}{2}} < q \leq k$.
   We can swap $q,q+2$ if both are in their starting positions for Starting Sequence 4. In other words, if $\sigma$ is a WCS with
   \[
     \sigma(q - 2) = q \quad \text{and} \sigma(q) = q+2,
   \]
   then the sequence $G_q(\sigma)$ given by 
    $$G_q(\sigma)(i) = \begin{cases} 
          q+2 & i = q-2 \\
          q & i = q \\
          \sigma(i) & \text{otherwise}
       \end{cases}$$
    is weakly consecutive as well.
\end{rul}

\begin{proof}
    Let $G_q := G_q(\sigma)$ for convenience. Note that $q > p + 2$.

    Fix an arbitrary $m \geq 2$, $i \neq j$, with $m \mid G_q(i)$ and $m \mid (i - j)$. If $m = q$, then $i = q$. Thus, $j$ is a multiple of $q$; however, this never happens since $2q > 2(p+2) = 2p + 4 > k$. Similarly, if $m = q + 2$, then $i = q - 2$. So $j$ is a multiple of $q - 2$. However, $2(q - 2) = 2q - 4 > 2(p+2) - 4 = 2p > k$, so this case also never happens.

    So now, we can assume that $m \notin \{ q, q + 2 \}$. But this means that $G_q(i) = \sigma(i)$, and since $\sigma$ is weakly consecutive, we know $m \mid \sigma(j)$. And similarly, $G_q$ only differs from $\sigma$ at the position of elements $q$ and $q + 2$; so $G_q(j) = \sigma(j)$ meaning $m \mid G_q(j)$, as desired.
\end{proof}

The first $p$ such that this rule gives sequences not otherwise achievable is $p = 11$, which for $k = 21$ we get the following WCS's:
\begin{gather*}
    3,4,5,6,7,8,9,10,\mathbf{13},12,\mathbf{11},14,15,16,17,18,19,20,21,\mathbf{2},\mathbf{1} \\
    3,4,5,6,7,8,9,10,\mathbf{13},12,\mathbf{11},14,15,16,\mathbf{19},18,\mathbf{17},20,21,\mathbf{2},\mathbf{1}
\end{gather*}

% It's worth noting that repeated application of this rule can potentially generate many WCS's of length $k$ if many such twin primes exist: for example, if Starting Sequence \ref{start:germain} $\sigma$ of length $k$ has primes $a,b$ with $a+2,b+2$ prime and $\lceil\frac{k}{2}\rceil < a,b \leq k$, then sequences $G_a(\sigma), G_b(\sigma)$ and $G_a(G_b(\sigma))=G_b(G_a(\sigma))$ are all weakly consecutive. For convenience we write $(G_a(G_b(\sigma))=G_{a,b}(\sigma)$.

We now discuss a rule with similar effect.

\begin{rul}[Prime Power Swaps] \label{rul:prime_power_swaps}
    Let $\sigma$ be a weakly consecutive sequence, and prime $p$ and integer $c \geq 1$ such that $p^c \leq k < p^c + p^{c-1}$. Then the sequence $S_{p,c}(\sigma)$ given by 
    $$S_{p,c}(\sigma)(i) = \begin{cases} 
          p^c & \sigma(i) = p^{x-1} \\
          p^{c-1} & \sigma(i) = p^c \\
          \sigma(i) & \text{otherwise}
       \end{cases}$$
    is weakly consecutive. 
\end{rul}

\begin{proof}[Proof of Rule \ref{rul:prime_power_swaps}]
    Let $\pi = S_{p, c}(\sigma)$ and define the sets $S_m^\sigma, S_m^\pi, T_m^\sigma, T_m^\pi$ as in Definition \ref{def:characterization}. By Lemma \ref{lem:characterization}, it suffices to show that $S_m^\pi = T_m^\pi$ for all $m \in [k]$. Fix such an $m \in [k]$.
    
    Case 1: $m = p^a$ for some integer $a$. Notice that $a \leq c$, since $k < p^{c} + p^{c-1} \leq p^{c+1}$. There are subcases:
    
    \begin{itemize}
      \item Case 1.1: $m = p^c$. Write $j = \pi^{-1}(p^c)$ for convenience. Then $T_m^\pi$ is a singleton set: namely, $T_m^\pi = \{ j \}$. It can't have any other elements since $p^c$ is the only multiple of $p^c$ less than $k$. Furthermore, $j \in S_m^\pi$ by definition. To show this is the only such element, suppose $i \in S_m^{\pi}$ is arbitrary. Note $i \equiv j \bmod p^c$, and since $j = \pi^{-1}(p^c) = \sigma^{-1}(p^{c-1})$, we also have $i \equiv \sigma^{-1}(p^{c-1}) \bmod p^{c-1}$ meaning $i \in S_{p^{c-1}}^\sigma$.
    
      Now, since $i \in S_{p^{c-1}}^\sigma$, this means that $i = j + ap^{c-1}$ for some integer $a$. The elements of $S_{p^{c-1}}^\sigma$ are uniformly spaced apart by a distance of $p^{c-1}$, so we know that $j, j + p^{c-1}, j + 2p^{c-1}, \dots, j + ap^{c-1}$ (or with negative signs if $a$ is negative) are all elements of $S_{p^{c-1}}^\sigma$. Thus $|S_{p^{c-1}}^\sigma| \geq |a| + 1$. But $|S_{p^{c-1}}^\sigma| = |T_{p^{c-1}}^\sigma| \leq p$; this is because there are only $p$ multiples of $p^{c-1}$ less than $k$ (since we assumed $(p+1)p^{c-1} > k$). Combining, we have $|a| \leq p - 1$.
    
      Finally, though, since $i \in S_m^\pi$ we know $i \equiv j \bmod p^c$. But $i = j + ap^{c-1}$ and $|a| \leq p-1$, we know $a$ is not large enough to reach another multiple of $p^c$ aside from $j$; thus $a = 0$ and $i = j$. So in fact, $S_m^\pi = \{ j \} = T_m^\pi$ as desired.
      \item Case 1.2: $m = p^a$ for $a < c$. Then $T_m^\pi = T_m^\sigma$, because the only multiples of $m$ which are moved are $p^{c-1}$ and $p^c$, and they are swapped with each other. To show that $S_m^\pi = S_m^\sigma$, notice that $\sigma^{-1}(p^c) \in T_m^{\sigma} = S_m^\sigma$ because $p^{c-1} \mid p^c$. Thus, $\sigma^{-1}(p^c) \equiv \sigma^{-1}(p^{c-1}) \bmod p^{c-1}$ and thus $\sigma^{-1}(p^c) \equiv \sigma^{-1}(p^{c-1}) \bmod m$. By substitution, $\pi^{-1}(p^c) \equiv \pi^{-1}(p^{c-1}) \bmod m$; and furthermore, no other multiples of $m$ moved from $\sigma$ to $\pi$. So in fact, $S_m^\pi = S_m^\sigma$; since $\sigma$ is weakly consecutive we have $S_m^\sigma = T_m^\sigma$, thus $S_m^\pi = T_m^\pi$.
    \end{itemize}
    
    Case 2: $m$ does not divide $p^{c-1}$. Then, $T_m^\pi = T_m^\sigma$, as no multiples of $m$ moved from $\sigma$ to $\pi$. For the same reason, $S_m^\pi = S_m^\sigma$, and thus $S_m^\pi = T_m^\pi$ by Lemma \ref{lem:characterization}.
\end{proof}

As with Twin Prime Swaps, we can perform Prime Power Swaps for many different valid primes if they exist for given $k$.

For example, if $k = 11$, then $(p,c) = (2,3)$, $(3,2)$, and $(11,1)$ are feasible, so we get the WCS's
\begin{gather*}
    1,2,3,4,5,6,7,8,9,10,11 \\
    1,2,3,\mathbf{8},5,6,7,\mathbf{4},9,10,11 \\
    1,2,\mathbf{9},4,5,6,7,8,\mathbf{3},10,11 \\
    \mathbf{11},2,3,4,5,6,7,8,9,10,\mathbf{1}
\end{gather*}
and others obtained by doing two or more Prime Power Swaps in succession.

Lastly, we have the simplest rule.

\begin{rul}[Trivial Reversal]\label{rul:triv rev}
    Let $\sigma$ be a weakly consecutive sequence. The sequence $R(\sigma)$ given by $R(\sigma)(i) = \sigma(k-i+1)$ for all $i$ is weakly consecutive.
\end{rul}

The proof is immediate.

For example, as the sequence $1,4,3,2$ is weakly consecutive, its reversal $2,3,4,1$ is also weakly consecutive.

\begin{remark} 
    $N(k) \geq 2$ for all $k \geq 2$, since both $(1,2,\dots,k)$ and its reversal $(k,\dots,2,1)$ are weakly consecutive. % I don't think we need to prove the 1,...,k case
    % Maybe split this into definition and theorem separately. Together right now because definition assumes p prime and c>= 2, which is kinda specific to this
\end{remark}

\begin{lemma}
The three rules commute; in other words, if applying Rule X followed by Rule Y preserves weak consecutiveness, then applying Rule Y followed by Rule X is also possible and also preserves weak consecutiveness.
\end{lemma}
\begin{proof}
By definition, the Trivial Reversal commutes with the other rules. Note that two Twin Prime Swaps commute because they concern different numbers in the WCS (the unique overlapping twin pairs, $(3,5)$ and $(5,7)$, cannot both lie in the swappable range for the same $k$). For the same reason, any two Prime Power Swaps commute.

It suffices to show that the Twin Prime Swap and the Prime Power Swap cannot interfere with each other. Twin Prime Swaps only affect primes to the first power. $2p-1$ cannot be prime if both $p,p+2$ are prime because one of $p,p+2$ is congruent to $1 \pmod 3$ and the other is congruent to $2 \pmod 3$. Then we have $p+p+2 \equiv 0\pmod 3$ and $2p-1\equiv 0\pmod 3$. Since the only Prime Power Swap allowed within the rules with $c=1$ is when $p\leq k<p+1$. Since $2p-1$ isn't prime and every other prime power of degree greater than one is in its original location we are free to do Prime Power Swaps.   
\end{proof}

\section{Prime Power Swaps} \label{sec:primepowerswaps}

We look with special depth into Prime Power Swaps since the number of $k$'s allowing Prime Power Swaps has a positive density and produces many new sequences, whereas the other rules either occur infrequently or do not produce many new sequences.

\paragraph{Notation:} Use $\lg(\cdot)$ for $\log_2(\cdot)$, while $\log(\cdot)$ denotes $\log_e(\cdot)$.

For a prime $p$ and a $c \in \Z^+$, define $I_{p,c} = [p^c, p^c + p^{c-1}) \cap \Z = \{ p^c, p^c + 1, \dots, p^c + p^{c-1} - 1 \}$; that is, the interval of sequence-lengths $k$ such that swapping $p^c$ and $p^{c-1}$ is permitted in Rule \ref{rul:prime_power_swaps}.

To determine the number of swaps, we ask how many such intervals a number $k$ is a part of; clearly it is always finite, as for each prime $p$ there is at most one $c$ with $k \in I_{p,c}$. Let us define this as $\#I(k)$; more formally,
\[
  \#I(k) = \left|\left\{ (p, c) : p \text{ prime}, c \in \Z^+, k \in I_{p,c} \right\}\right|.
\]
Thus $\#I(k) \leq \pi(k)$ with the usual prime counting function $\pi$. We can say something stronger, in fact; for each $c$, there is at most one $p$ with $n \in I_{p,c}$. This is because for fixed $c$ and distinct primes $p < q$, we have $p^c + p^{c-1} \leq (p+1)^c \leq q^c$ by applying the binomial theorem to $(p+1)^c$, implying $I_{p,c} \cap I_{q,c} = \emptyset$.

When $k \in I_{p,c}$, then $2^c \leq p^c \leq k$. Thus, $c$ can range from $1$ to $\lg(n)$, so $\#I(k) \leq \lg(k)$. Note that by Rule~\ref{rul:prime_power_swaps}, $N(k) \geq 2^{\#I(k)}$. 

\subsection{Lower bounds on \texorpdfstring{$\#I(k)$}{\#I(k)} and \texorpdfstring{$N(k)$}{N(k)}}
One can ask how large of a number $k$ is required to attain a desired number of intervals (and thus Prime Power Swaps), $n$; that is, $\#I(k) \geq n$. It is not immediately clear that $\#I(k)$ is unbounded, but the following result establishes a lower bound on the growth rate of $\#I(k)$:

\begin{theorem} \label{thm:asymptote}
  Fix an $n \in \Z^+$. There exists a $k \leq 2^{2^{2n \lg(n) (1 + o(1))}}$ with $\#I(k) \geq n$.
\end{theorem}

As in the introduction, let $N(k)$ denote the number of weakly consecutive sequences of length $k$. Since $N(k) \geq 2^{\#I(k)}$, we have the following corollary:

\begin{corollary} \label{cor:nk}
  Fix an $n \in \mathbb{Z}^+$. There exists a $k \leq 2^{n^{2 \lg(\lg (n)) (1 + o(1))}}$ with $N(k) \geq n$.
\end{corollary}

Our focus will be on proving Theorem~\ref{thm:asymptote}.

\textit{Proof of Theorem~\ref{thm:asymptote}.} Fix an $n \in \Z^+$, and consider the first $2n$ primes, $2 = p_1 < p_2 < \dots < p_{2n}$. For convenience, define $p = p_{2n}$. Define $\varepsilon := \frac{\log 
\left(1 + \frac1p\right)}{\log p}$ and consider the set of numbers $\{x_i\}_{1 \leq i \leq 2n-1}$ defined by $x_i := \frac{\log p_i}{\log p}$. By simultaneous Dirichlet approximation \cite{Dir-encyc-maths}, there exists a non-zero integer $C \leq \left\lceil \frac{1}{\varepsilon^{2n-1}} \right\rceil$ and integers $\{z_i\}_{1 \leq i \leq 2n-1}$ such that $|Cx_i - z_i| \leq \varepsilon$ for all $i$. Furthermore, we can assume without loss of generality that $C$ is positive (else we can take $-C$ and $-z_i$).

Now, we prove two useful lemmas.

\begin{lemma} \label{lem:intervalone}
  If $0 \leq Cx_i - z_i \leq \varepsilon$, then $p^C \in I_{p_i, z_i}$.
\end{lemma}

\begin{proof}
  We need to show $p_i^{z_i} \leq p^C$ and $p^C < p_i^{z_i} + p_i^{z_i - 1}$.

  For the first inequality, note that $z_i \leq Cx_i$ by our assumption, meaning
  \[
    p_i^{z_i} \leq p_i^{Cx_i} = \left(p_i^{\frac{\log(p_i)}{\log(p)}}\right)^C = p^C.
  \]
  For the second inequality, we use the fact that $Cx_i \leq z_i + \varepsilon$, so that
  \[
    p^C = p_i^{Cx_i} \leq p_i^{z_i + \varepsilon} = p_i^{z_i} p_i^{\varepsilon}.
  \]
  Note that $\varepsilon = \frac{\log 
  \left(1 + \frac1p\right)}{\log p} < \frac{\log 
  \left(1 + \frac1p\right)}{\log {p_i}}$, since $p_i < p$. Thus $p^\varepsilon < 1 + \frac1p < 1 + \frac{1}{p_i}$, meaning
  \[
    p^C < p_i^{z_i} \left(1 + \frac{1}{p_i}\right) = p_i^{z_i} + p_i^{z_i - 1}
  \]
  as desired.
\end{proof}

\begin{lemma} \label{lem:intervaltwo}
  If $0 \leq z_i - Cx_i \leq \varepsilon$, then $p^C + p^{C-1} \in I_{p_i, z_i}$.
\end{lemma}

\begin{proof}
  We need to show $p_i^{z_i} \leq p^C + p^{C-1}$ and $p^C + p^{C-1} < p_i^{z_i} + p_i^{z_i-1}$.

  For the first inequality, we know $z_i \leq Cx_i + \varepsilon$, meaning that
  \[
    p_i^{z_i} \leq p_i^{Cx_i + \varepsilon} = p^C p_i^{\varepsilon}.
  \]
  Note as before that $\varepsilon < \frac{\log 
  \left(1 + \frac1p\right)}{\log {p_i}}$, meaning that $p_i^\varepsilon < 1 + \frac1p$ and thus $p_i^{z_i} \leq p^C\left(1 + \frac1p\right) = p^C + p^{C-1}$. So we have the first inequality.

  For the second inequality, we first use the fact that $Cx_i \leq z_i$, meaning in fact
  \[
    p^C = p_i^{Cx_i} \leq p_i^{z_i}.
  \]
  Multiplying both sides by $1 + \frac1p$, we have
  \[
    p^C + p^{C-1} \leq p_i^{z_i} \left(1 + \frac1p \right).
  \]
  However, we know that $1 + \frac1p < 1 + \frac{1}{p_i}$, meaning in fact $p^C + p^{C - 1} < p_i^{z_i}\left(1 + \frac1{p_i}\right) = p_i^{z_i} + p_i^{z_i-1}$, as desired.
\end{proof}

Now that we have both of these lemmas, notice that we have a set of $2n$ numbers, $x_i$, and integers $C$, $z_i$, with $0 \leq |Cx_i - z_i| \leq \varepsilon$ for all $i$. It must be the case that either $0 \leq Cx_i - z_i \leq \varepsilon$ for at least $n$ values of $i$, or $0 \leq z_i - Cx_i \leq \varepsilon$ for at least $n$ values of $i$. In the former case, Lemma~\ref{lem:intervalone} applies to each such $i$, and $\#I(p^C) \geq n$; otherwise Lemma~\ref{lem:intervaltwo} applies and $\#I(p^C + p^{C-1}) \geq n$.

Finally, we must take into account how large of a quantity $p^C$ or $p^C(1 + 1/p)$ could be. We know $C \leq \left\lceil \frac{1}{\varepsilon^{2n}} \right\rceil$ and $p = p_{2n+1}$; define $f(n)$ to be
\[
  f(n) = p_{2n+1}^{\frac{1}{\varepsilon^{2n}} + 1}(1 + 1/p_{2n+1}) = p_{2n+1}^{\left( \frac{\log(p_{2n+1})}{\log(1 + 1/p_{2n+1})}  \right)^{2n}}(1 + 1/p_{2n+1}).
\]
Clearly $f(n)$ is an upper bound on $p^C$ and $p^C + p^{C-1}$ for any such $n$, meaning $\#I(k) \geq n$ for some $k \leq f(n)$. Thus we focus on bounding the size of $f(n)$.

Once again for simplicity we write $p = p_{2n+1}$, despite that $p$ varies with $n$. Note that
\[
  \lg f(n) = \left(\frac{\log(p)}{\log(1 + 1/p)}\right)^{2n} \lg(p) + \lg(1 + 1/p) = \left(\frac{\log(p)}{\log(1 + 1/p)}\right)^{2n} \lg(p) (1 + o(1))
\]
and
\[
  \lg \lg f(n) = 2n \lg \left(\frac{\log(p)}{\log(1 + 1/p)}\right) + \lg \lg(p) + o(1).
\]
Note that $\log(1 + 1/p) \geq 1/(2p)$ so that
\begin{align*}
  \lg \lg f(n) &\leq 2n \lg ( 2p\log(p) ) + \lg \lg (p) + \lg(1 + o(1)) \\
  &= 2n(1 + \lg(p) + \lg(\log(p))) + \lg \lg (p) + \lg(1 + o(1)) \\
  &= 2n \lg(p) (1 + o(1)).
\end{align*}
Furthermore, $\lg(p) \sim \lg(n)$ as $p = p_{2n+1} \sim \frac{2n}{\log n}$. Thus we have $f(n) \leq 2n \lg(n) (1 + o(1))$, or $f(n) \leq 2^{2^{2n \lg(n) (1 + o(1))}}$, as desired. \qed

\begin{remark}
  Our best known \emph{lower} bound on how large $k$ must be to achieve $\#I(k) \geq n$ is only singly exponential; since $\#I(k) \leq \lg(k)$, we know $k \geq 2^n$.
\end{remark}

\subsection{Upper bounds on \texorpdfstring{$\liminf_{k \to \infty} \#I(k)$}{liminf \#I(k)}}
In this section, we show that there are infinitely many $k$ admitting very few Prime Power Swaps.

For this section, it will be useful to consider the setup of a few random variables. Fix some $a, b \in \mathbb{R}$ with $a < b$, and let $X \sim \mathcal{U}(a, b)$ be a uniform real-valued random variable. Let $K = \lfloor \exp(X) \rfloor$ so that $K$ is an integer random variable with $\lfloor \exp(a) \rfloor \leq K \leq \lfloor \exp(b) \rfloor$. Lastly, we let $R$ be the random variable counting the number of intervals $I_{p, c}$ that $K$ is a member of. We can write
\[
    R = \sum_{p \text{ prime}} R_p
\]
where $R_p = 1$ if $K \in I_{p, c}$ for some $c$, and $0$ otherwise.

We make the following observations.

\begin{lemma} \label{lem:loguniform}
    For a fixed prime $p$ and positive integer $c$, we have $K \in I_{p, c}$ if and only if $X \in \left[\log(p^c), \log(p^c) + \log\left(1+\frac{1}{p}\right)\right)$.
\end{lemma}

\begin{proof}
    Suppose $X \in \left[\log(p^c), \log(p^c) + \log\left(1+\frac{1}{p}\right)\right)$. Since $X \geq \log(p^c)$ then $K \geq \lfloor \exp(\log(p^c)) \rfloor = p^c$. Similarly, if $X < \log(p^c) + \log\left(1+\frac{1}{p}\right) = \log(p^c + p^{c-1})$, then since $\exp(\log(p^c + p^{c-1}))$ is an integer, $K = \lfloor\exp(X)\rfloor < p^c + p^{c-1}$.

    Similarly, suppose $K \in I_{p, c}$. Then $K \geq p^c$, meaning $\exp(X) \geq p^c$ so $X \geq \log(p^c)$. Furthermore, $K < p^c + p^{c-1}$, meaning $\exp(X) < p^c + p^{c-1}$, one again using the fact that $p^c + p^{c-1}$ is an integer. Thus $X < \log(p^c) + \log\left(1 + \frac1p\right)$.
\end{proof}

To make the following arguments less cumbersome, let $L_{p, c} = \left[c\log(p), c\log(p) + \log\left(1+\frac{1}{p}\right)\right)$. So $X \in L_{p,c} \Leftrightarrow K \in I_{p, c}$.

\begin{lemma} \label{lem:uniformintervals}
    Let $x, y \in \mathbb{R}$ be constants with $x > y \geq 0$, and let $L \subseteq \mathbb{R}$ be the set
    \[
        L = \bigcup_{c \in \Z} \left[ cx, cx + y \right)
    \]
    That is, $L$ is the union of evenly spaced, uniform-width intervals on the real line. Then, $\Pr[X \in L] \leq \frac{y}{x} + \frac{2y}{b-a}$.
\end{lemma}

The proof of Lemma~\ref{lem:uniformintervals} is deferred to the appendix, but it is immediately applied to the following lemma.
\begin{lemma} \label{lem:rp_bound}
    Fix a prime $p$. Then, $\E[R_p] \leq \frac{\log\left(1 + \frac{1}{p}\right)}{\log(p)} + \frac{2\log\left(1 + \frac{1}{p}\right)}{b - a}$.
\end{lemma}

\begin{proof}
    Consider the set $L_p = \bigcup_{c \in \Z^+} L_{p,c}$. This set is a subset of the $L$ defined in Lemma~\ref{lem:uniformintervals} with $x = \log(p)$ and $y = \log\left(1 + \frac{1}{p}\right)$. Thus, $\Pr[X \in L_p] \leq \frac{\log\left(1 + \frac{1}{p}\right)}{\log(p)} + \frac{2\log\left(1 + \frac{1}{p}\right)}{b - a}$. But notice that, by Lemma~\ref{lem:loguniform}, $X$ being in some $L_{p, c}$ is equivalent to $K$ being in some $I_{p, c}$. Thus, $R_p = 1$ if and only if $X \in L_p$, and otherwise $R_p = 0$. The claim follows.
\end{proof}

Now, we have the following theorem.
\begin{theorem}
    The number of prime power intervals $\#I(k)$ is at most $1$ infinitely often. That is,
    \[
        \liminf_{k \to \infty} \#I(k) \leq 1.
    \]
\end{theorem}
\begin{proof}
    Let $N$ be a large perfect square, and set $a = \frac12 \log N$ and $b = \log N$. Note that since $X \in [a, b]$ always, we have $K \leq \exp(b) = N$. This means that whatever the value of $K$ is, it cannot be in any intervals $I_{p, c}$ with $p > N$. So in fact,
    \[
        R = \sum_{\text{prime }p \leq N} R_p.
    \]
    Let $R' = \sum_{{\substack{p \text{ prime} \\ 3 \leq p \leq N}}} R_p$. For the purpose of this theorem, we will only consider the primes $p \geq 3$ and will ignore $p =2$; this is because we will be able to show $\E[R'] < 1$ but not $\E[R] < 1$. Now, we have
    \begin{align*}
        \E[R'] &= \sum_{\substack{p \text{ prime} \\ 3 \leq p \leq N}} \E[R_p] &\text{by linearity} \\
        &\leq \sum_{\substack{p \text{ prime} \\ 3 \leq p \leq N}} \left( \frac{\log\left(1 + \frac{1}{p}\right)}{\log(p)} + \frac{2\log\left(1 + \frac{1}{p}\right)}{\frac12 \log N} \right) &\text{by Lemma~\ref{lem:rp_bound}} \\
        &\leq \sum_{\substack{p \text{ prime} \\ 3 \leq p \leq N}} \left( \frac{1}{p\log(p)} + \frac{4}{p \log N} \right) &\text{since $\log(1+x) \leq x$} \\
        &= \sum_{\substack{p \text{ prime} \\ 3 \leq p \leq N}}  \frac{1}{p\log(p)} + \sum_{\substack{p \text{ prime} \\ 3 \leq p \leq N}} \frac{4}{p \log N} \\
        &= \sum_{\substack{p \text{ prime} \\ 3 \leq p \leq N}}  \frac{1}{p\log(p)} + \frac{4}{\log N}\sum_{\substack{p \text{ prime} \\ 3 \leq p \leq N}} \frac{1}{p}
    \end{align*}
    We bound the sums separately. For the first sum, note that the infinite version of the sum, with all the primes, converges to \cite{mathar2018digits}
    % https://oeis.org/A137245
    \[
        \sum_{p \text { prime}} \frac{1}{p \log p} = 1.636616323\dots < 1.637
    \]
    so that
    \[
        \sum_{\substack{p \text{ prime} \\ 3 \leq p \leq N}}  \frac{1}{p\log(p)} < \sum_{p \text { prime}} \frac{1}{p \log p} - \frac{1}{2 \log 2} < 0.92
    \]
    Now, for the second sum, we note that
    \[
        \sum_{\substack{p \text{ prime} \\ p \leq N}} \frac{1}{p} = \log\log(N) + O(1)
    \]
    so that
    \[
        \frac{4}{\log N} \sum_{\substack{p \text{ prime} \\ 3 \leq p \leq N}} \frac{1}{p} \leq \frac{4 \log \log (N) + O(1)}{\log N} < 0.01
    \]
    for a sufficiently large choice of $N$. Summing together, we see that $\E[R'] < 1$ for a sufficiently large $N$. Since $R'$ is a nonnegative integer, then, we know that there exists some choice $X = x \in [a, b]$ so that $R' = 0$. Thus, there is a value $K = k$, with $\sqrt{N} \leq k \leq N$, such $k \notin I_{p, c}$ for any primes $p \geq 3$. Thus $\#I(k) \leq 1$.

    Since we can increase $N$ without bound to obtain infinitely many such values of $k$, the claim follows.
\end{proof}

\section{Conjectures and future work} \label{sec:further}
This paper has shown that the notion of weakly consecutive sequence, while simple to define, leads to some intricate behaviors. Foremost among our conjectures is that the starting sequences and rules described in Section \ref{sec:rules} are sufficient to describe all WCS's. 

\begin{conjecture}\label{conj:main}
    Any WCS can be derived from one of the four starting sequences by applying the three rules. Moreover, for $k \geq 12$, the starting sequence and rule applications are unique up to reordering the rules.
\end{conjecture}
We have verified this by computer calculation up to $k = 500$.  The uniqueness is false for $k \leq 11$. For example, when $k = 11$, there is a Twice Twin Prime sequence (starting sequence 3)
\[
  (1,\mathbf{10},3,4,\mathbf{7},6,\mathbf{5},8,9,\mathbf{2},11),
\]
but it can also be derived from the consecutive sequence by three Prime Power Swaps and a Trivial Reversal:
\begin{align*}
  (1,2,3,4,5,6,7,8,9,10,11)
  &\mapsto (\mathbf{11},2,3,4,5,6,7,8,9,10,\mathbf{1}) \\
  &\mapsto (11,2,\mathbf{9},4,5,6,7,8,\mathbf{3},10,1) \\
  &\mapsto (11,2,9,\mathbf{8},5,6,7,\mathbf{4},3,10,1) \\
  &\mapsto (1,10,3,4,7,6,5,8,9,2,11).
\end{align*}
If Conjecture \ref{conj:main} is true, it implies certain properties of WCS's by simply checking that they hold on the starting sequences and are preserved by the rules. A seemingly basic one is as follows:
\begin{conjecture}\label{conj:1 at end}
  Any WCS either begins or ends with the term $1$ (in other words, $\sigma(1) = 1$ or $\sigma(k) = 1$).
\end{conjecture}
Surprisingly, we have not been able to prove this conjecture by working straight from the definition. If true, then together with Lemma \ref{lem:div_slice}, it implies the following:
\begin{conjecture}\label{conj:m near end}
  Let $\sigma$ be a WCS. Then each $m < k/2$ appears as either one of the first $m$ or last $m$ terms of the sequence; in other words
  \[
    \sigma^{-1}(m) \in [1,m] \cup [k-m+1,k].
  \]
\end{conjecture}
Also, we have an explanation for the patterns in Table \ref{tab:N(k)}:
\begin{conjecture}\label{conj:power of 2}
  The number $N(k)$ of WCS's is always a power of $2$ unless $k$ is of the form $2p - 1$ where $p$, $2p + 1$, and $p + 2$ are all primes. In the latter case, $N(k)$ is the sum of two powers of $2$.
\end{conjecture}
If Conjecture \ref{conj:main} is true, then Conjecture \ref{conj:power of 2} will be proved by observing that the rules can always be applied to starting sequences 1, 1 and 2, or 1 and 3 (starting sequences 2 and 3 never occur for the same $k$. When Starting Sequence 4 appears, it generates a possibly richer family of swaps.

Recall that \emph{Dickson's conjecture} \cite{Dickson,Prime,BatemanHorn} states that any finite set of linear forms
\[
  a_1 + b_1 x, \ldots, a_r + b_r x, \quad a_i, b_i \in \mathbb{Z}
\]
are simultaneously prime for infinitely many positive integers $x$, in the absence of a congruence obstruction to this.

If Dickson's conjecture is true then Starting Sequence 4 occurs infinitely often. When Starting Sequence 4 appears, then Rule \ref{rul:germain} is activated. If, moreover, the first Hardy–Littlewood conjecture is true (\cite{HardyLittlewood}; \cite[\textsection A9]{Guy}, then the number of twin prime pairs in $[k/2, k]$ grows like a constant times $k/\lg^2 k$, giving
\[
  N(k) \geq 2^{c k/\lg^2 k}.
\]
By contrast, using the other rules, we can get at most $O(k)$ WCS's since $I(k) \leq \lg k$. So the $N(k)$-values arising from Starting Sequence 4 are expected to be of \emph{much} larger order than average, a trend only beginning to be seen in Table \ref{tab:N(k)}. We are thus led to conjecture:
\begin{conjecture}
  If all WCS's are listed in order of increasing length $k$, then asymptotically 100\% of them are derived from Starting Sequence 4.
\end{conjecture}

On the other hand, for 100\% of $k$, none of Starting Sequences 2, 3, 4 apply, and if Conjecture \ref{conj:main} holds, then $N(k) = 2^{n+1}$ where $n$ is the number of applicable Prime Power Swaps. It is reasonable to believe that the primes behave independently, leading to the following expectation:
\begin{conjecture}
    The value $N(k) = 2$ (minimal for $k > 1$) occurs for infinitely many $k$'s. Specifically, the density of $k$'s such that $N(k) = 2$ is given by
    \[
      \delta = \prod_{p \text{ prime}} \left(1 - \frac{\log\left(1 + \frac{1}{p}\right)}{\log(p)}\right) \approx 0.19.
    \]
\end{conjecture}
The same heuristics suggest that each value $N(k) = 2^{n+1}$ should occur a positive proportion of the time.

\appendix
\section{Deferred proofs}
\begin{proof}[Proof of Lemma~\ref{lem:uniformintervals}]
    Let $X \sim \mathcal{U}(a, b)$ be a uniform random variable. Let $x, y \in \mathbb{R}$ be constants with $x > y \geq 0$, and let $L \subseteq \mathbb{R}$ be the set
    \[
        L = \bigcup_{c \in \Z} \left[ cx, cx + y \right)
    \]
    Since $X$ is a uniform random variable, we have
    \[
        \Pr[X \in L] = \frac{\mu(L \cap [a, b])}{\mu(a, b)} = \frac{\mu(L \cap [a, b])}{b - a},
    \]
    where $\mu$ denotes the Lebesgue measure. Note that
    \[
        \mu(L \cap [a, b]) \leq y \cdot | \{ c \in \Z : [cx, cx + y) \cap [a, b] \neq \emptyset \} |,
    \]
    since each interval $[cx, cx + y) \cap [a, b]$ has measure at most $y$. Let $c_{\min}$ and $c_{\max}$ denote the minimum and maximum value, respectively, of $c \in \Z$ where $[cx, cx + y)$ has a non-empty intersection with $[a, b]$. We know $c_{\max} x \leq b$ and $c_{\min} x + y \geq a$, and furthermore, the number of values of $c$ with intersecting intervals is $c_{\max} - c_{\min} + 1$. Applying to the above inequality, we have
    \begin{align*}
        \mu(L \cap [a, b]) &\leq y \cdot (c_{\max} - c_{\min} + 1) \leq y \left(\frac{b}{x} - \frac{a-y}{x} + 1 \right) \\
        &= y \cdot \frac{b-a}{x} + y\left(1 + \frac{y}{x}\right) \leq y \cdot \frac{b-a}{x} + 2y.
    \end{align*}
    Thus, we have
    \[
        \Pr[X \in L] \leq \frac{y \cdot \frac{b-a}{x} + 2y}{b-a} = \frac{y}{x} + \frac{2y}{b-a},
    \]
    as desired.
\end{proof}
\section{Acknowledgements}
% The authors do not need to thank each other in the paper.
%First I'd like to thank Chris Seiler and Andrew Knowles for agreeing to collaborate with me and for allowing our research project to happen. They also did so much work on the project. 
We would like to thank our advisors Evan O'Dorney, Prasad Tetali, and Aliaksei Semchankau. We also want to thank everyone we've talked to about this problem, including but not limited to Tolson Bell, Soumil Mukherjee, Jason Chadwick, Arpita Nag, and Anna Dietrich. 
\bibliography{WCS}
\bibliographystyle{plain}
\end{document}